\documentclass[12pt,twoside]{amsart}
\usepackage{amssymb,amscd,amsxtra,calc,mathrsfs}
\usepackage{amsmath}
\usepackage{xypic}

\title{Openness results for uniform K-stability}
\author{Kento Fujita} 
\date{\today}
\address{Research Institute for Mathematical Sciences, Kyoto University, Kyoto 606-8502, Japan}
\email{fujita@kurims.kyoto-u.ac.jp}

\newcommand{\pr}{\mathbb{P}}

\newcommand{\Z}{\mathbb{Z}}
\newcommand{\Q}{\mathbb{Q}}
\newcommand{\R}{\mathbb{R}}

\newcommand{\A}{\mathbb{A}}
\newcommand{\G}{\mathbb{G}}

\newcommand{\ND}{\operatorname{N}^1}

\newcommand{\Nef}{\operatorname{Nef}}

\newcommand{\Eff}{\operatorname{Eff}}

\newcommand{\Spec}{\operatorname{Spec}}

\newcommand{\id}{\operatorname{id}}

\newcommand{\DF}{\operatorname{DF}}

\newcommand{\ord}{\operatorname{ord}}

\newcommand{\vol}{\operatorname{vol}}

\newcommand{\NA}{\operatorname{NA}}

\newcommand{\sC}{\mathcal{C}}
\newcommand{\sO}{\mathcal{O}}

\newcommand{\sX}{\mathcal{X}}
\newcommand{\sY}{\mathcal{Y}}
\newcommand{\sL}{\mathcal{L}}

\newcommand{\sM}{\mathcal{M}}

\newcommand{\sZ}{\mathcal{Z}}
\newcommand{\sD}{\mathcal{D}}

\newtheorem{thm}{Theorem}[section]
\newtheorem{lemma}[thm]{Lemma}
\newtheorem{proposition}[thm]{Proposition}
\newtheorem{corollary}[thm]{Corollary}

\theoremstyle{definition}
\newtheorem{definition}[thm]{Definition}
\newtheorem{remark}[thm]{Remark}

\newtheorem*{ack}{Acknowledgments}

\begin{document}

\maketitle 

\begin{abstract}
Assume that a projective variety together with a polarization is uniformly K-stable. 
If the polarization is canonical or anti-canonical, then the projective variety is 
uniformly K-stable with respects to any polarization 
sufficiently close to the original polarization. 
\end{abstract}

\setcounter{tocdepth}{1}
\tableofcontents

\section{Introduction}\label{intro_section}

In this article, we work over an arbitrary algebraically closed field $\Bbbk$ of 
characteristic zero. A \emph{variety} is assumed to be a connected, reduced, 
separated and of finite type scheme over $\Spec\Bbbk$. For the 
minimal model program, we refer the readers to \cite{KoMo} and \cite{SingBook}. 

In this article, we show the following result: 

\begin{thm}[{see Corollaries \ref{fano2_cor} and \ref{gt2_cor}}]\label{mainthm}
Let $(X, \Delta)$ be a projective slc pair and $L$ be an ample $\Q$-line bundle on $X$. 
Assume that $((X, \Delta), L)$ is uniformly K-stable. If $L=K_X+\Delta$ or 
$L=-(K_X+\Delta)$, then there exists a Euclidean open neighborhood 
$U\subset\ND(X)_\R$ of $L$ such that $((X, \Delta), L')$ is uniformly K-stable 
for any $\Q$-line bundle $L'$ with $L'\in U$. 
\end{thm}

Let $(X, \Delta)$ be a projective slc pair and $L$ be an ample $\Q$-line bundle on $X$. 
Motivated by the fundamental works \cite{tian, don, sz1, sz2, dervan, BHJ}, we are 
interested in whether $((X, \Delta), L)$ is \emph{uniformly K-stable} or not. 
However, many basic properties of uniform K-stability remain unknown. 
For example, it is expected that the uniform K-stability of $((X, \Delta), L)$ implies 
the uniform K-stability of $((X, \Delta), L')$ for any ample $\Q$-line bundle $L'$ such that 
$L'$ is very close to $L$ in $\ND(X)_\R$. 
In fact, LeBrun and Simanca showed in \cite[Corollary 2]{LS} that, if 
$(M, J, \omega)$ is a compact constant scalar curvature K\"ahler manifold with 
the automorphism group of $M$ semisimple, then there exists a Euclidean open 
neighborhood $U\subset H^{1,1}(M, \R)$ of the class $[\omega]$ such that each 
element in $U$ can be represented by the K\"ahler form of 
a constant scalar curvature K\"ahler metric. 
The problem is hard to prove in general. 
Theorem \ref{mainthm} gives a partial affirmative answer of the problem. 

The idea of the proof of Theorem \ref{mainthm} is simple. 
Firstly we perturb both the boundary and the polarization. Secondly we perturb 
only the boundary. For the first step, when $L=-(K_X+\Delta)$, we use Odaka's theorem 
\cite{annals} and a valuative 
criterion for uniform K-stability of log Fano pairs established in \cite{fjta, li, fjtb, fjt17} 
(see Theorem \ref{fanolog_thm}); when $L=K_X+\Delta$, we use the result 
on the uniform bounds of Donaldson-Futaki invariants divided by certain norms 
\cite[Corollary 9.3]{BHJ} (see Theorem \ref{odk_thm} \eqref{odk_thm1}). 
In particular, when $L=-(K_X+\Delta)$, we will show the following theorem. 

\begin{thm}[{see Theorem \ref{fanolog_thm} \eqref{fanolog_thm1}}]\label{fanolog_intro_thm}
Let $(X, \Delta)$ be an $n$-dimensional log Fano pair and set $L:=-(K_X+\Delta)$. 
Assume that $((X, \Delta), L)$ is uniformly K-stable. Set 
\[
\varepsilon:=\frac{\delta(X,\Delta)-1}{n\cdot\delta(X,\Delta)+n+1}.
\]
Then, for any effective $\Q$-Cartier $\Q$-divisor $B$ on $X$ with $\varepsilon L-B$ 
ample $($resp., nef$)$, $(X, \Delta+B)$ is a log Fano pair and $((X, \Delta+B), L-B)$ is 
uniformly K-stable $($resp., K-semistable$)$. 
$($For the definitions, see \S 2, \S 4 and \S 5.$)$
\end{thm}

For the second step, we will show the following proposition. 

\begin{proposition}[{=Proposition \ref{perturb_prop}}]\label{main_prop}
Let $(X, \Delta)$ be an $n$-dimensional projective demi-normal pair, $L$ be an ample 
$\Q$-line bundle and $N$ be an effective and nef $\Q$-divisor on $X$. Then, for any 
semiample demi-normal test configuration $(\sX,\sL)/\pr^1$ of $(X, L)$, we have 
\[
n\mu_N(L) J^{\NA}(\sX, \sL)\geq \DF_{\Delta+N}(\sX, \sL)-\DF_\Delta(\sX, \sL).
\]
$($For the definitions, see \S \ref{pre_section}--\ref{uniform_section}.$)$
\end{proposition}

As corollaries of the discussion in the first step and Proposition \ref{main_prop}, we get 
Theorem \ref{mainthm}. Moreover, in the proof, we show the following theorem: 

\begin{thm}[{=Theorem \ref{W_thm}. 
See also \cite{weinkove, song-weinkove}}]\label{main2_thm}
Let $(X, \Delta)$ be an $n$-dimensional projective slc pair with $n\geq 2$ 
and let $L$ be an ample 
$\Q$-line bundle on $X$. Assume that $\mu_{K_X+\Delta}(L)>0$ and 
\[
\frac{n^2}{n^2-1}\mu_{K_X+\Delta}(L)L-(K_X+\Delta)
\]
is ample $($resp., nef$)$, where 
\[
\mu_{K_X+\Delta}(L):=\frac{\left(L^{\cdot n-1}\cdot (K_X+\Delta)\right)}{(L^{\cdot n})}.
\] 
Then $((X, \Delta), L)$ is uniformly K-stable $($resp., 
K-semistable$)$. 
\end{thm}

The article is organized as follows. In Section \ref{pre_section}, we recall basic theories 
of demi-normal varieties and we see some properties of the cones of pseudo-effective 
or nef divisors in the $\R$-tensors of the N\'eron-Severi groups. In Section 
\ref{tc_section}, we recall the definition of test configurations. Moreover, we establish 
a fundamental theory of \emph{demi-normal test configurations} of demi-normal 
polarized pairs. Thanks to the theory, we do not need to consider almost trivial 
test configurations in the sense of \cite[Definition 2.9]{BHJ}. In Section 
\ref{uniform_section}, we define the notions of uniform K-stability and K-semistability. 
Moreover, we recall fundamental results of Odaka \cite{calabi, genRT, annals}. 
In Section \ref{fano_section}, we recall the theory established in \cite{fjta, li, fjtb, fjt17}. 
Moreover, we show in Theorem \ref{fanolog_thm} that, if $((X, \Delta), -(K_X+\Delta))$ 
is uniformly K-stable, then $((X, \Delta+B), -(K_X+\Delta+B))$ is also uniformly K-stable 
for any effective and very small $B$. In Section \ref{perturb_section}, we prove 
Theorem \ref{mainthm} by showing Proposition \ref{main_prop}.

\begin{ack}
This work started while the author enjoyed the AIM workshop ``Stability and Moduli 
Spaces". The author thanks the organizers and staff for the stimulating environment. 
During and after the workshop, the author learned many motivations and backgrounds 
from Doctors Giulio Codogni and Ruadha\'i Dervan. Especially, they helped the author 
to improve Theorem \ref{main2_thm} (see also Remark \ref{CD_rmk}). 
The author is supported by JSPS KAKENHI Grant Number JP16H06885.
\end{ack}

\section{Preliminaries}\label{pre_section}

\subsection{Demi-normal pairs}\label{demi_section}

We recall the notion of demi-normal varieties. A standard reference is 
\cite[\S 5]{SingBook}.

\begin{definition}[{\cite[\S 5.1]{SingBook}}]\label{demi_dfn}
\begin{enumerate}
\renewcommand{\theenumi}{\arabic{enumi}}
\renewcommand{\labelenumi}{(\theenumi)}
\item\label{demi_dfn1}
An equi-dimensional variety $X$ is said to be a \emph{demi-normal variety} if 
$X$ satisfies Serre's $S_2$ condition and any codimension one point $\eta\in X$ 
satisfies that either $\sO_{X,\eta}$ is regular or double normal crossing. 
\item\label{demi_dfn2}
Let $X$ be a demi-normal variety and let $\nu\colon\bar{X}\to X$ be the normalization. 
The \emph{conductor ideal} of $X$ is defined to be 
\[
\mathfrak{cond}_X:=\mathcal{H}om_{\sO_X}(\nu_*\sO_{\bar{X}}, \sO_X)\subset\sO_X. 
\]
This ideal sheaf can be seen as an ideal sheaf $\mathfrak{cond}_{\bar{X}}\subset
\sO_{\bar{X}}$, named the \emph{conductor ideal} of $\bar{X}/X$. We define 
\[
D_X:=\Spec_X(\sO_X/\mathfrak{cond}_X),\quad 
D_{\bar{X}}:=\Spec_{\bar{X}}(\sO_{\bar{X}}/\mathfrak{cond}_{\bar{X}}), 
\]
and say them the \emph{conductor divisor} of $X$, the \emph{conductor divisor} of 
$\bar{X}/X$, respectively. 
Let $\bar{D}_{\bar{X}}$ be the normalization of $D_{\bar{X}}$. Then 
we can get the natural Galois involution 
$\tau_X\colon\bar{D}_{\bar{X}}\to\bar{D}_{\bar{X}}$. 
\item\label{demi_dfn3}
Let $X$ be a demi-normal variety. A \emph{divisor} (resp., a \emph{$\Q$-divisor}) on $X$ 
is a formal finite $\Z$-linear (resp., $\Q$-linear) sum $\sum_ia_i\Delta_i$ such that 
each $\Delta_i$ is an irreducible and reduced codimension one subvariety of $X$ with 
$\Delta_i\not\subset D_X$. 
\item\label{demi_dfn4}
A pair $(X, \Delta)$ is said to be a \emph{demi-normal pair} (resp., a \emph{normal pair})
if $X$ is a demi-normal variety (resp., a normal variety) and $\Delta$ is a $\Q$-divisor 
on $X$ such that each coordinates belongs to the set $[0,1]$, and $K_X+\Delta$ is 
a $\Q$-Cartier $\Q$-divisor on $X$, where $K_X$ is the canonical divisor on $X$. 
\item\label{demi_dfn5}
For any $n$-dimensional demi-normal projective variety $X$, 
for any ample $\Q$-line bundle $L$ on $X$, and for any $\Q$-divisor $\Delta$ on $X$, 
we set 
\[
\mu_\Delta(L):=\frac{(L^{\cdot n-1}\cdot \Delta)}{(L^{\cdot n})}.
\]
\end{enumerate}
\end{definition}

We use the following proposition later. 

\begin{proposition}[{\cite[Proposition 5.3]{SingBook}}]\label{glue_prop}
Let $X$ be a demi-normal variety, let $\bar{X}$, $\bar{D}_{\bar{X}}$, $\tau_X$ be 
as in Definition \ref{demi_dfn} \eqref{demi_dfn2}. Then the triplet 
$(\bar{X}, \bar{D}_{\bar{X}}, \tau_X)$ uniquely determines $X$. 
\end{proposition}

\begin{definition}[{see \cite[\S 5.2]{SingBook} for example}]\label{discrep_dfn}
Let $(X, \Delta)$ be a demi-normal pair and let $\nu\colon\bar{X}\to X$ be the 
normalization. Set $\bar{\Delta}:=\nu_*^{-1}\Delta$ and let $D_{\bar{X}}$ be the 
conductor divisor of $\bar{X}/X$. Then $(\bar{X}, D_{\bar{X}}+\bar{\Delta})$ is a 
(possibly non-connected) normal pair. 

Let $F$ be a prime divisor \emph{over} $\bar{X}$, 
that is, there exists a projective birational 
morphism $\sigma\colon Y\to\bar{X}$ with $Y$ a (possibly non-connected) normal variety 
and $F$ a prime divisor \emph{on} $Y$. 
\begin{enumerate}
\renewcommand{\theenumi}{\arabic{enumi}}
\renewcommand{\labelenumi}{(\theenumi)}
\item\label{discrep_dfn1}
We set the \emph{log discrepancy} $A_{(X, \Delta)}(F)$ of $(X, \Delta)$ along $F$ as 
\[
A_{(X, \Delta)}(F):=1+\ord_F(K_Y-\sigma^*(K_{\bar{X}}+D_{\bar{X}}+\bar{\Delta})).
\]
\item\label{discrep_dfn1}
The pair $(X, \Delta)$ is said to be a \emph{semi log canonical pair} (\emph{slc pair}, for 
short) if $A_{(X, \Delta)}(F)\geq 0$ holds for any prime divisor $F$ over $\bar{X}$. 
\item\label{discrep_dfn1}
The pair $(X, \Delta)$ is said to be a \emph{Kawamata log terminal pair} (\emph{klt pair}, for short) if $A_{(X, \Delta)}(F)>0$ holds for any prime divisor $F$ over $\bar{X}$. 
If $(X, \Delta)$ is a klt pair, then $(X, \Delta)$ must be a normal pair. 
\item\label{discrep_dfn1}
The pair $(X, \Delta)$ is said to be a \emph{log Fano pair} if $(X, \Delta)$ is a projective 
klt pair and $-(K_X+\Delta)$ is an ample $\Q$-Cartier $\Q$-divisor. 
\end{enumerate}
\end{definition}

\subsection{On cones of Cartier divisors}\label{cone_section}

For a projective variety $X$, the $\R$-tensor of the N\'eron-Severi group $\ND(X)_\R$ 
of $X$ is a finite dimensional vector space over $\R$. Moreover, the nef cone 
$\Nef(X)\subset\ND(X)_\R$ is a closed and strongly convex cone. Moreover, if $X$ is 
normal, then the pseudo-effective cone $\overline{\Eff}(X)
\subset\ND(X)_\R$ is also a closed and strongly convex cone. 
See \cite{L1} and \cite{N} for example.

\begin{lemma}\label{fanocone_lem}
Let $(X, \Delta)$ be a log Fano pair. Then $\Nef(X)\subset\ND(X)_\R$ is spanned 
by the classes of finitely many semiample Cartier divisors. Moreover, $\overline{\Eff}(X)
\subset\ND(X)_\R$ is spanned by the classes of finitely many effective Cartier divisors. 
\end{lemma}

\begin{proof}
By \cite[Corollary 1.4.3]{BCHM}, there exists a small, projective and birational morphism 
$\sigma\colon\tilde{X}\to X$ with $\tilde{X}$ $\Q$-factorial. Since 
$(\tilde{X}, \sigma_*^{-1}\Delta)$ is a klt pair and $-(K_{\tilde{X}}+\sigma^{-1}_*\Delta)$ 
is nef and big, there exists a $\Q$-divisor $\tilde{\Delta}\geq\sigma^{-1}_*\Delta$ such 
that $(\tilde{X}, \tilde{\Delta})$ is a log Fano pair (see also \cite[Lemma 2.1]{fjta}). 
Thus $\tilde{X}$ is a Mori dream space in the sense of \cite{HK} by 
\cite[Corollary 1.3.2]{BCHM}. By \cite[Definition 1.10 (2) and Proposition 1.11 (2)]{HK}, 
$\Nef(\tilde{X})$ is spanned by the classes of finitely many semiample Cartier divisors 
and $\overline{\Eff}(\tilde{X})$ is spanned by the classes of finitely many effective Cartier 
divisors. 
It is easy to see that, under the natural linear inclusion 
\[
\sigma^*\colon \ND(X)_\Q\hookrightarrow\ND(\tilde{X})_\Q, 
\]
we have $\Nef(X)=\Nef(\tilde{X})\cap\ND(X)_\R$ and $\overline{\Eff}(X)=
\overline{\Eff}(\tilde{X})\cap\ND(X)_\R$.
Thus we get the assertion. 
\end{proof}

The following proposition is intrinsically trivial. 

\begin{proposition}\label{norm_prop}
Let $X$ be an $n$-dimensional demi-normal projective variety, $L$ be an ample 
$\Q$-line bundle on $X$. Fix any norm $\|\cdot\|$ on $\ND(X)_\R$. 
For any $\varepsilon>0$, there exists $\delta>0$ such that the set 
\[
\sC^1_\delta:=\{t(L+\xi)\in\ND(X)_\R\,\,|\,\,t\in\R_{\geq 0},\, \xi\in\ND(X)_\R,\, 
\|\xi\|\leq \delta\}
\]
is a subset of $\sC^2_\varepsilon\cap\sC^3_\varepsilon$, where 
\begin{eqnarray*}
\sC^2_\varepsilon&:=&\{t(L-a)\in\ND(X)_\R\,\,|\,\,t\in\R_{\geq 0},\, a\in\Nef(X)_\R,\, 
\|a\|\leq \varepsilon\},\\
\sC^3_\varepsilon&:=&\{t(L+a)\in\ND(X)_\R\,\,|\,\,t\in\R_{\geq 0},\, a\in\Nef(X)_\R,\, 
\|a\|\leq \varepsilon\}.
\end{eqnarray*}
\end{proposition}

\begin{proof}
Set $\rho:=\dim_\R\ND(X)_\R$. Since $L$ is ample, there exist linearly independent 
$a_1,\dots,a_\rho\in\ND(X)_\R$ with $a_1,\dots,a_\rho\in\Nef(X)$ and there exist 
$t_1,\dots,t_\rho\in\R_{>0}$ such that $L=\sum_{i=1}^\rho t_ia_i$ and 
$\sum_{i=1}^\rho t_i=1$. Set $t_0:=\min_it_i\in\R_{>0}$. 
We may assume that the norm $\|\cdot\|$ is given by 
\[
\left\|\sum_{i=1}^\rho s_ia_i\right\|:=\sum_{i=1}^\rho |s_i|.
\]
Take any $\delta\in(0, t_0)$ and any $\xi=\sum_{i=1}^\rho \xi_ia_i\in\ND(X)_\R$ with 
$\|\xi\|\leq \delta$ (i.e., $\sum_{i=1}^\rho |\xi_i|\leq \delta$). We note that 
$t_0^{-1}\delta t_i\geq |\xi_i|$ for any $i$. 
Moreover, we have 
\begin{eqnarray*}
L+\xi &=&(1+t_0^{-1}\delta)\left(L-\frac{1}{1+t_0^{-1}\delta}\sum_{i=1}^\rho
(t_0^{-1}\delta t_i-\xi_i)a_i\right), \\
L+\xi &=&(1-t_0^{-1}\delta)\left(L+\frac{1}{1-t_0^{-1}\delta}\sum_{i=1}^\rho
(t_0^{-1}\delta t_i+\xi_i)a_i\right).
\end{eqnarray*}
The classes 
\[
\frac{1}{1+t_0^{-1}\delta}\sum_{i=1}^\rho(t_0^{-1}\delta t_i-\xi_i)a_i,\quad
\frac{1}{1-t_0^{-1}\delta}\sum_{i=1}^\rho(t_0^{-1}\delta t_i+\xi_i)a_i
\]
are nef and those norms are bounded below by $(t_0^{-1}+1)\delta/(1-t_0^{-1}\delta)$. 
As a consequence, if we take $\delta\in(0,t_0)$ with $(t_0^{-1}+1)\delta/(1-t_0^{-1}\delta)
\leq \varepsilon$, then we get 
$\sC^1_\delta\subset\sC_\varepsilon^2\cap\sC_\varepsilon^3$. 
\end{proof}

\section{Demi-normal test configurations}\label{tc_section}

In this section, we see a fundamental theory for test configurations of demi-normal 
polarized varieties. In Section \ref{tc_section}, we always assume that 
$X$ is a demi-normal projective variety and $L$ is an ample $\Q$-line bundle on $X$. 

\begin{definition}\label{tc_dfn}
\begin{enumerate}
\renewcommand{\theenumi}{\arabic{enumi}}
\renewcommand{\labelenumi}{(\theenumi)}
\item\label{tc_dfn1}
(see \cite{tian, don})
A \emph{semiample test configuration} (resp., an \emph{ample test configuration}) 
$(\sX, \sL)/\pr^1$ of $(X, L)$ consists of: 
\begin{itemize}
\item
a projective variety $\sX$ together with a flat morphism $\sX\to\pr^1$, 
\item
a $\pi$-semiample (resp., a $\pi$-ample) $\Q$-line bundle $\sL$ on $\sX$, 
\item
a holomorphic $\G_m$-action $\G_m\curvearrowright(\sX, \sL)$ 
commuting with the multiplicative 
action $\G_m\curvearrowright\pr^1$, and they satisfy that
\item
$(\sX\setminus\sX_0,\sL|_{\sX\setminus\sX_0})$ is $\G_m$-equivariantly isomorphic to 
$(X\times(\pr^1\setminus\{0\}), p_1^*L)$ with the natural $\G_m$-action, where 
$\sX_0$ is the scheme-theoretic fiber of $\sX\to\pr^1$ 
at $0\in\pr^1$ and $p_1\colon X\times(\pr^1\setminus\{0\})\to X$ 
is the first projection. 
\end{itemize}
\item\label{tc_dfn2}
Let $(\sX, \sL)/\pr^1$ be a semiample test configuration of $(X, L)$. $(\sX, \sL)/\pr^1$ 
is said to be a \emph{demi-normal test configuration} if 
\begin{enumerate}
\renewcommand{\theenumii}{\roman{enumii}}
\renewcommand{\labelenumii}{(\theenumii)}
\item\label{tc_dfn21}
$\sX$ is a demi-normal variety, and 
\item\label{tc_dfn22}
for any generic point $\eta\in\sX_0$, the local ring $\sO_{\sX, \eta}$ is regular.
\end{enumerate}
If $X$ is normal, then a demi-normal test configuration is called a \emph{normal test 
configuration}. (Note that, for the definition of a normal test configuration, the condition 
\eqref{tc_dfn22} follows immediately from the condition \eqref{tc_dfn21}.)
\item\label{tc_dfn23}
For a semiample, demi-normal test configuration $(\sX, \sL)/\pr^1$ and for a 
$\Q$-divisor $\Delta$ on $X$, let $\Delta_{\sX}$ be the $\Q$-divisor on $\sX$ defined by 
the closure of $\Delta\times(\pr^1\setminus\{0\})$ under the canonical isomorphism 
$\sX\setminus\sX_0\simeq X\times(\pr^1\setminus\{0\})$. 
\end{enumerate}
\end{definition}

We recall the notion of the partial normalization of test configurations which is important 
for our study. 

\begin{definition}[{\cite[\S 3]{genRT}, \cite[\S 5]{annals}}]\label{partial_dfn}
Let $(\sX, \sL)/\pr^1$ be a semiample (resp., an ample) test configuration of $(X, L)$. 
Let $i\colon\sX\setminus\sX_0\hookrightarrow\sX$ be the inclusion and let 
$\nu\colon\bar{\sX}\to\sX$ be the normalization. Set 
\[
\sX^{p\nu}:=\Spec_\sX(i_*\sO_{\sX\setminus\sX_0}\cap\nu_*\sO_{\bar{\sX}})
\xrightarrow{p\nu} \sX.
\]
From the definition, $\nu$ factors through 
$\bar{\sX}\to\sX^{p\nu}\xrightarrow{p\nu}\sX$. Of course, 
$(\sX^{p\nu}, p\nu^*\sL)/\pr^1$ is also a semiample (resp., an ample) test configuration 
of $(X, L)$. We call it the \emph{partial normalization} of $(\sX, \sL)/\pr^1$. 
\end{definition}

\begin{proposition}\label{ring_prop}
The above $(\sX^{p\nu}, p\nu^*\sL)/\pr^1$ is a demi-normal test configuration. 
\end{proposition}

\begin{proof}
For any generic point $\eta\in\bar{\sX}_0$, the morphism $\bar{\sX}\to\sX^{p\nu}$ is 
an isomorphism at $\eta$ by \cite[Lemma 3.9]{genRT}. 
Thus it is enough to check that $\sX^{p\nu}$ satisfies Serre's $S_2$ condition. 
Take any $x\in\sX_0$. Take an affine open subvariety $x\in U\simeq\Spec R$
around $x\in\sX$. Then, around over $x\in\sX$, $\sX^{p\nu}$ is written as the spectrum 
of the following $\Bbbk$-algebra 
\[
R^{p\nu}:=R[t^{-1}]\cap\bar{R}, 
\]
where $t\in R$ is the non-homogeneous coordinate of $\pr^1$ and $\bar{R}$ is the 
integral closure of $R$ in the total quotient ring $K$ of $R$. 

For any $S\in\{R^{p\nu}$, $R[t^{-1}]$, $\bar{R}\}$ and for any $a\in K$, let us set 
\[
\sD_S(a):=\{s\in S\,|\, as\in S\}
\]
as in \cite[(2.3)]{HH}. This is an ideal of $S$. Moreover, let us consider the 
extension 
\[
\widetilde{S}:=\{a\in K\,|\, \text{ht}\sD_S(a)\geq 2\}
\]
of $S$. (We set $\text{ht} S:=+\infty$.)
By \cite[Proposition (2.4)]{HH} (see also \cite[Remark 1.4]{ciuperca}), 
the ring $S$ satisfies Serre's $S_2$ condition if and only if $\widetilde{S}=S$ holds. 
In particular, we have $\widetilde{R[t^{-1}]}=R[t^{-1}]$ and $\widetilde{\bar{R}}=\bar{R}$. 
Take any $a\in\widetilde{R^{p\nu}}$. Since $\text{ht}\sD_{R^{p\nu}}(a)\geq 2$, we have 
$\text{ht}\sD_{R^{p\nu}}(a)\bar{R}\geq 2$. Since $\sD_{R^{p\nu}}(a)\bar{R}\subset
\sD_{\bar{R}}(a)$, we have $a\in\widetilde{\bar{R}}=\bar{R}$. On the other hand, 
we know that $R^{p\nu}[t^{-1}]=R[t^{-1}]$. Since 
$\text{ht}\sD_{R^{p\nu}}(a)R[t^{-1}]\geq 2$ and $\sD_{R^{p\nu}}(a)R[t^{-1}]\subset
\sD_{R[t^{-1}]}(a)$, we have $a\in\widetilde{R[t^{-1}]}=R[t^{-1}]$. These imply that 
$a\in R^{p\nu}$. Therefore, $R^{p\nu}$ satisfies Serre's $S_2$ condition. 
\end{proof}

\begin{remark}\label{s2_rmk}
Let $(\sX, \sL)/\pr^1$ be a semiample test configuration of $(X, L)$. Assume that 
the local ring $\sO_{\sX, \eta}$ is regular for any generic point $\eta\in\sX_0$. Then 
the partial normalization $\sX^{p\nu}$ is nothing but the \emph{$S_2$-ification} of 
$\sX$ (in the sense of \cite[Proposition (5.10.10) and (5.10.11)]{G}, 
\cite[(2.3)]{HH}, \cite[Definition 6.20]{V} and \cite[\S 4]{LN}) by Proposition \ref{ring_prop}. 
In the word of \cite[Definition 5.1]{SingBook}, $\sX^{p\nu}$ is the 
\emph{demi-normalization} of $\sX$. In particular, any demi-normal test configuration is 
equal to its partial normalization. 
\end{remark}

\begin{lemma}\label{normalization_lemma}
Let $(\sX, \sL)/\pr^1$, $(\sY, \sM)/\pr^1$ be test configurations of $(X, L)$. 
Assume that there exists a $\G_m$-equivariant birational morphism 
$\phi\colon\sX\to\sY$ over $\pr^1$. Then $\phi$ lifts to the partial normalizations 
$\phi^{p\nu}\colon\sX^{p\nu}\to\sY^{p\nu}$. 
\end{lemma}

\begin{proof}
Let $\nu_\sX\colon\bar{\sX}\to\sX$, $\nu_\sY\colon\bar{\sY}\to\sY$ be the 
normalizations, and let $i_\sX\colon\sX\setminus\sX_0\hookrightarrow\sX$, 
$i_\sY\colon\sY\setminus\sY_0\hookrightarrow\sY$  be the inclusions. 
We have a natural commutative diagram
\[\xymatrix{
\bar{\sX} \ar[d]_{\nu_\sX} \ar[r]^{\bar{\phi}}& \bar{\sY}  \ar[d]^{\nu_\sY} \\
\sX\ar[r]_\phi & \sY.
}\]
The morphism $\bar{\sX}\to\sX\times_\sY\bar{\sY}$ over $\sX$ induces a 
homomorphism
\[
\phi^*(\nu_\sY)_*\sO_{\bar{\sY}}\to(\nu_\sX)_*\sO_{\bar{\sX}}
\]
of coherent $\sO_\sX$-algebras. On the other hand, we have a natural homomorphism 
\[
\phi^*(i_\sY)_*\sO_{\sY\setminus\sY_0}\to(i_\sX)_*\sO_{\sX\setminus\sX_0}
\]
of quasi-coherent $\sO_\sX$-algebras. Thus we get a natural homomorphism 
\[
\phi^*\left((\nu_\sY)_*\sO_{\bar{\sY}}\cap(i_\sY)_*\sO_{\sY\setminus\sY_0}\right)
\to(\nu_\sX)_*\sO_{\bar{\sX}}\cap(i_\sX)_*\sO_{\sX\setminus\sX_0}
\]
of coherent $\sO_\sX$-algebras. The homomorphism induces a morphism 
$\sX^{p\nu}\to\sX\times_\sY\sY^{p\nu}$ over $\sX$. 
\end{proof}

\begin{corollary}\label{afs_cor}
Let $(\sX, \sL)/\pr^1$ be a semiample, demi-normal test configuration of $(X,L)$. 
Let $\phi\colon (\sX, \sL)\to (\sY, \sM)/\pr^1$ be the \emph{ample model} of 
$(\sX, \sL)/\pr^1$ in the sense of \cite[Definition 2.16]{BHJ}, i.e., 
$\phi$ is a projective birational morphism with $\phi_*\sO_{\sX}=\sO_{\sY}$ and 
$(\sY, \sM)/\pr^1$ is an ample test configuration of $(X, L)$ with 
$\phi^*\sM\sim_\Q\sL$. Then $(\sY, \sM)/\pr^1$ is a demi-normal test configuration 
of $(X, L)$. 
\end{corollary}

\begin{proof}
Take any generic point $\eta\in\sY_0$. Since $\phi_*\sO_{\sX}=\sO_{\sY}$ and 
the morphism $\sX\to\pr^1$ is flat, the morphism 
$\phi$ is an isomorphism over a neighborhood of $\eta$ by Zariski's main theorem 
(see \cite[Proposition 4.4.2]{liu} for example). 
Thus it is enough to check that $\sY$ satisfies Serre's $S_2$ condition. 
Let $p\nu\colon\sY^{p\nu}\to\sY$ be the partial normalization. 
By Remark \ref{s2_rmk}, the partial normalization of $\sX$ is $\sX$ itself. 
Thus, together with Lemma \ref{normalization_lemma}, we get the following 
commutative diagram
\[\xymatrix{
& \sY^{p\nu}  \ar[d]^{p\nu} \\
\sX \ar[ur]^{\phi'} \ar[r]_\phi & \sY.
}\]
Note that the composition of the inclusions 
\[
\sO_{\sY}\hookrightarrow p\nu_*\sO_{\sY^{p\nu}}\hookrightarrow p\nu_*\phi'_*\sO_{\sX}
=\sO_{\sY}
\]
is an identity. Thus $\sO_{\sY}=p\nu_*\sO_{\sY^{p\nu}}$. This implies that $\sY^{p\nu}
=\sY$. 
\end{proof}

\begin{definition}\label{trivial_dfn}
Let $(\sX, \sL)/\pr^1$ be a semiample, demi-normal test configuration of $(X, L)$ and 
let $\phi\colon(\sX, \sL)\to(\sY, \sM)/\pr^1$ be the ample model of $(\sX, \sL)/\pr^1$. 
(By Corollary \ref{afs_cor}, $(\sY, \sM)/\pr^1$ is an ample, demi-normal 
test configuration of $(X, L)$.) $(\sX, \sL)/\pr^1$ is said to be a \emph{trivial test 
configuration} of $(X, L)$ if $(\sY, \sM)/\pr^1$ is $\G_m$-equivariantly isomorphic to 
$(X\times\pr^1, p_1^*L)$ with the natural $\G_m$-action, where 
$p_1\colon X\times\pr^1\to X$ is the first projection. 
\end{definition}

\begin{definition}\label{nrmlztn_dfn}
Let $(\sX, \sL)/\pr^1$ be a semiample (resp., an ample), 
demi-normal test configuration of $(X, L)$. Let $\nu\colon\bar{X}\to X$, $\nu\colon
\bar{\sX}\to \sX$ be the normalizations. Then $(\bar{\sX}, \nu^*\sL)/\pr^1$ is a 
(possibly non-connected) semiample (resp., ample), normal test configuration of 
$(\bar{X}, \nu^*L)$. We call this \emph{the associated normal test configuration} 
of $(\bar{X}, \nu^*L)$.
\end{definition}

The following proposition is useful in order to check whether a given demi-normal 
test configuration is trivial or not. 

\begin{proposition}\label{trivial_prop}
Let $(\sX, \sL)/\pr^1$ be an ample, demi-normal test configuration of $(X, L)$ and 
let $(\bar{\sX}, \nu^*\sL)/\pr^1$ be the associated ample, normal test configuration 
of $(\bar{X}, \nu^*L)$. Then $(\sX, \sL)/\pr^1$ is the trivial test configuration if and 
only if $($any connected component of$)$ $(\bar{\sX}, \nu^*\sL)/\pr^1$ is so. 
\end{proposition}

\begin{proof}
Let $D_{\bar{X}}\subset\bar{X}$ be the conductor divisor of $\bar{X}/X$, let 
$\bar{D}_{\bar{X}}$ be its normalization and let $\tau_X\colon\bar{D}_{\bar{X}}\to
\bar{D}_{\bar{X}}$ be the natural involution as in Definition \ref{demi_dfn} \eqref{demi_dfn1}. 
Similarly, let $D_{\bar{\sX}}\subset\bar{\sX}$ be the conductor divisor of $\bar{\sX}/\sX$, 
let $\bar{D}_{\bar{\sX}}$ be its normalization and let 
$\tau_{\sX}\colon\bar{D}_{\bar{\sX}}\to\bar{D}_{\bar{\sX}}$ be the natural involution. 
Assume that $(\bar{\sX}, \nu^*\sL)/\pr^1$ is the trivial test configuration. 
Then $\bar{\sX}$ is $\G_m$-equivariantly isomorphic to $\bar{X}\times\pr^1$. 
Thus $D_{\bar{\sX}}$ is $\G_m$-equivariantly isomorphic to $D_{\bar{X}}\times\pr^1$. 
Moreover, the involutions $\tau_X\times \id_{\pr^1}=\tau_{X\times\pr^1}\colon
{\bar{D}}_{\bar{X}}\times\pr^1\to{\bar{D}}_{\bar{X}}\times\pr^1$ must be equal to 
$\tau_{\sX}\colon\bar{D}_{\bar{\sX}}\to\bar{D}_{\bar{\sX}}$ 
under the above $\G_m$-equivariant isomorphism since 
$\tau_{X\times\pr^1}$ and $\tau_{\sX}$ are equal over $\pr^1\setminus\{0\}$. 
Thus we get the assertion from Proposition \ref{glue_prop}. 
\end{proof}

\section{Uniform K-stability}\label{uniform_section}

We recall the definition of Donaldson-Futaki invariants and various stability conditions. 

\begin{definition}\label{DF_dfn}
Let $(X, \Delta)$ be an $n$-dimensional projective demi-normal pair, let $L$ be an 
ample $\Q$-line bundle on $X$ and let $(\sX, \sL)/\pr^1$ be a semiample, demi-normal 
test configuration of $(X, L)$. Let 
\[\xymatrix{
& \sY  \ar[dl]_\Theta \ar[dr]^\Pi & \\
\sX \ar@{-->}[rr] & & X\times\pr^1
}\]
be the \emph{partial normalization} of the graph of the natural birational map 
$\sX\dashrightarrow X\times\pr^1$. Let $p_1\colon X\times\pr^1\to X$ be the 
first projection. 
\begin{enumerate}
\renewcommand{\theenumi}{\arabic{enumi}}
\renewcommand{\labelenumi}{(\theenumi)}
\item\label{DF_dfn1}(\cite[Definition 7.6]{BHJ})
We set 
\[
J^{\NA}(\sX, \sL):=\frac{1}{(L^{\cdot n})}\left((\Theta^*\sL\cdot
\Pi^*p_1^*L^{\cdot n})-\frac{1}{n+1}(\sL^{\cdot n+1})\right).
\]
\item\label{DF_dfn2}(\cite{wang}, \cite{genRT} and \cite[Definition 3.17]{BHJ})
We set the \emph{log Donaldson-Futaki invariant} as 
\[
\DF_\Delta(\sX, \sL):=\frac{1}{(L^{\cdot n})}
\left((\sL^{\cdot n}\cdot(K_{\sX/\pr^1}+\Delta_{\sX}))-\frac{n}{n+1}
\mu_{K_X+\Delta}(L)(\sL^{\cdot n+1})\right), 
\]
where 
$K_{\sX/\pr^1}$ is $K_{\sX}$ minus the pullback of $K_{\pr^1}$. 
\end{enumerate}
\end{definition}

\begin{proposition}\label{invariant_prop}
Let $(X, \Delta)$, $L$, $(\sX, \sL)/\pr^1$ be as in Definition \ref{DF_dfn}. 
\begin{enumerate}
\renewcommand{\theenumi}{\arabic{enumi}}
\renewcommand{\labelenumi}{(\theenumi)}
\item\label{invariant_prop1}
For any projective birational morphism $\phi\colon (\sZ, \phi^*\sL)\to(\sX, \sL)$
between demi-normal test configurations of $(X, L)$, we have 
\begin{eqnarray*}
J^{\NA}(\sX, \sL) &=& J^{\NA}(\sZ, \phi^*\sL),\\ 
\DF_\Delta(\sX, \sL) &=& \DF_\Delta(\sZ, \phi^*\sL). 
\end{eqnarray*}
\item\label{invariant_prop2}
Let $(\bar{\sX}, \nu^*\sL)/\pr^1$ be the associated normal test configuration of 
$(\sX, \sL)/\pr^1$. Then we have 
\begin{eqnarray*}
J^{\NA}(\sX, \sL) &=& J^{\NA}(\bar{\sX}, \nu^*\sL),\\ 
\DF_\Delta(\sX, \sL) &=& \DF_{D_{\bar{X}}+\bar{\Delta}}(\bar{\sX}, \nu^*\sL), 
\end{eqnarray*}
where $\bar{\Delta}:=\nu^{-1}_*\Delta$ on $\bar{X}$ and $D_{\bar{X}}$ is the 
conductor divisor of $\bar{X}/X$. 
\item\label{invariant_prop3}
We have $J^{\NA}(\sX, \sL)\geq 0$. Moreover, equality holds if and only if 
$(\sX, \sL)/\pr^1$ is a trivial test configuration of $(X, L)$. 
\end{enumerate}
\end{proposition}

\begin{proof}
\eqref{invariant_prop1} is trivial. See \cite{dervan, BHJ} for example. 

\eqref{invariant_prop2} 
For the normalization $\nu\colon \bar{X}\to X$, we know that $\nu^*K_X=K_{\bar{X}}
+D_{\bar{X}}$. The assertion immediately follows from this fact. 

\eqref{invariant_prop3}
By \eqref{invariant_prop1}, after replacing its ample model (see Corollary \ref{afs_cor}), 
we may assume that $(\sX, \sL)/\pr^1$ is an ample, demi-normal test configuration. 
If $\sX$ is normal, then the assertion follows from \cite[Theorem 1.3]{dervan} and 
\cite[Theorem 7.9]{BHJ}. For a general case, it follows from the normal case, 
\eqref{invariant_prop2} and Proposition \ref{trivial_prop}. 
\end{proof}

\begin{definition}[{see \cite{BHJ, dervan} for example}]\label{K_dfn}
Let $(X, \Delta)$ be a projective demi-normal pair and $L$ be an ample $\Q$-line bundle 
on $X$. 
\begin{enumerate}
\renewcommand{\theenumi}{\arabic{enumi}}
\renewcommand{\labelenumi}{(\theenumi)}
\item\label{K_dfn1}
We say that $((X, \Delta), L)$ is \emph{uniformly K-stable} if 
there exists $\delta\in(0,1)$ such that for any semiample, demi-normal test configuration 
$(\sX,\sL)/\pr^1$ of $(X, L)$, the inequality $\DF_\Delta(\sX, \sL)\geq \delta \cdot 
J^{\NA}(\sX,\sL)$ holds. 
\item\label{K_dfn2}
We say that $((X, \Delta), L)$ is \emph{K-stable} (resp., \emph{K-semistable}) if, 
for any non-trivial, semiample, demi-normal test configuration 
$(\sX,\sL)/\pr^1$ of $(X, L)$, the inequality 
$\DF_\Delta(\sX, \sL)>0$ (resp., $\geq 0$) holds. 
\item\label{K_dfn3}
We say that $((X, \Delta), L)$ is \emph{K-polystable}  if 
$((X, \Delta), L)$ is K-semistable and, the equality $\DF_\Delta(\sX, \sL)=0$ for an ample, 
demi-normal test configuration $(\sX, \sL)/\pr^1$ of $(X, L)$ implies that 
$(\sX\setminus\sX_\infty, \Delta_{\sX}|_{\sX\setminus\sX_\infty})
\simeq(X\times\A^1, \Delta\times\A^1)$, where $\sX_\infty$ is the scheme-theoretic 
fiber of $\sX\to\pr^1$ at $\infty\in\pr^1$. 
\end{enumerate}
\end{definition}

\begin{remark}\label{imply_rmk}
\begin{enumerate}
\renewcommand{\theenumi}{\arabic{enumi}}
\renewcommand{\labelenumi}{(\theenumi)}
\item\label{imply_rmk1}
From Proposition \ref{invariant_prop}, uniform K-stability implies K-stability. 
Moreover, it is obvious that K-stability implies K-polystability and K-polystability 
implies K-semistability. 
\item\label{imply_rmk2}
After Li and Xu found a pathological example \cite[Example 3]{LX}, the notion of 
\emph{test configurations trivial in codimension $2$}, 
\emph{almost trivial test configurations} was given in \cite[Definition 1]{stoppa}, 
\cite[Definition 3.3]{pams}, respectively. See also \cite[Definition 2.9]{BHJ}. 
Thanks to Propositions \ref{ring_prop}, \ref{trivial_prop} and Corollary \ref{afs_cor}, 
the definitions of K-stability and K-semistability in our sense coincide with the ones 
in the senses of \cite[Corollary 3.11]{genRT} and \cite[Definition 3.11]{BHJ}. 
In particular, we do not need to consider almost trivial test 
configurations in order to test K-stability of $((X, \Delta), L)$ for demi-normal pairs 
$(X, \Delta)$. 
\end{enumerate}
\end{remark}

The following theorems are important for the studies of K-stability. 

\begin{thm}[{\cite[Theorem 1.2]{annals}}]\label{annals_thm}
Let $(X, \Delta)$ be a projective demi-normal pair and let $L$ be an ample $\Q$-line 
bundle. If $((X, \Delta), L)$ is K-semistable, then $(X, \Delta)$ is an slc pair. 
\end{thm}

\begin{thm}\label{odk_thm}
Let $(X, \Delta)$ be an $n$-dimensional projective slc pair. 
\begin{enumerate}
\renewcommand{\theenumi}{\arabic{enumi}}
\renewcommand{\labelenumi}{(\theenumi)}
\item\label{odk_thm1}
$($\cite[Corollary 9.3]{BHJ}, see also \cite[Theorem 1.1 (i)]{calabi}, 
\cite[Theorem 1.2 (ii)]{dervan}$)$ Assume that 
$L:=K_X+\Delta$ is ample. Then, for any semiample, demi-normal test configuration 
$(\sX, \sL)/\pr^1$ of $(X, L)$, we have 
\[
\DF_\Delta(\sX, \sL)\geq \frac{1}{n}\cdot J^{\NA}(\sX,\sL).
\] 
In particular, 
$((X, \Delta), L)$ is uniformly K-stable. 
\item\label{odk_thm2}
$($\cite[Theorem 1.1 (ii)]{calabi} and \cite[Corollary 9.4]{BHJ}, see also 
\cite[Theorem (iii)]{dervan}$)$
Assume that $K_X+\Delta\equiv 0$ and let $L$ be an arbitrary ample $\Q$-line bundle 
on $X$. Then $((X, \Delta), L)$ is K-semistable. Moreover, 
$((X, \Delta), L)$ is uniformly K-stable if and only if $(X, \Delta)$ is a klt pair. 
\item\label{odk_thm3}
$($\cite[Theorem 1.3]{annals} and \cite[Corollary 9.6]{BHJ}$)$
Assume that $L:=-(K_X+\Delta)$ is ample. If $((X, \Delta), L)$ is K-semistable, 
then $(X, \Delta)$ must be a klt pair $($i.e., $(X, \Delta)$ is a log Fano pair$)$. 
\end{enumerate}
\end{thm}

\begin{proof}
Only \eqref{odk_thm1} is unknown. However, the assertion immediately follows from 
Proposition \ref{invariant_prop} \eqref{invariant_prop2} and \cite[\S 9]{BHJ} by 
considering the associated normal test configurations. 
\end{proof}

We sometimes use the following negativity lemma.

\begin{lemma}[{see \cite[Lemma 6.14]{BHJ}}]\label{negativity_lem}
Let $X$ be an $n$-dimensional demi-normal projective variety, let $L$ be an ample 
$\Q$-line bundle on $X$ and let $(\sX, \sL)/\pr^1$ be a semiample, demi-normal 
test configuration of $(X, L)$. Assume that $D$ is a $\Q$-Cartier $\Q$-divisor 
on $\sX$ supported on $\sX_0$, and $\sM_1,\dots,\sM_{n-1}$ be $\Q$-Cartier 
$\Q$-divisors on $\sX$ such that each $\sM_i$ is nef over $\pr^1$. Then we have 
$(\sM_1\cdots\sM_{n-1}\cdot D^{\cdot 2})\leq 0$.
\end{lemma}

\begin{proof}
Follows immediately from \cite[Lemma 6.14]{BHJ}
after taking the normalization of $\sX$. 
\end{proof}

\section{Uniform K-stability of log Fano pairs}\label{fano_section}

In this section, we always assume that $(X, \Delta)$ is an $n$-dimensional log Fano pair 
and $L:=-(K_X+\Delta)$. 
We recall the theory established in \cite{fjta, li, fjtb, fjt17}. 
More precisely, there is a simple criterion to test whether $((X, \Delta), L)$ is 
uniformly K-stable or not. We recall this in this section. 

\begin{definition}\label{vol_dfn}
Let $F$ be a prime divisor over $X$. Fix a projective birational morphism 
$\sigma\colon Y\to X$ such that $Y$ is normal and $F$ is a prime divisor on $Y$. 
\begin{enumerate}
\renewcommand{\theenumi}{\arabic{enumi}}
\renewcommand{\labelenumi}{(\theenumi)}
\item\label{vol_dfn1}
For any Cartier divisor $M$ on $X$ and for any $x\in\R_{\geq 0}$, let 
$H^0(X, M-xF)$ be the sub $\Bbbk$-vector space of $H^0(X, M)$ defined by
\[
H^0(X, M-xF):=H^0(Y, \sigma^*M-xF)\subset H^0(Y, \sigma^*M)
\]
under the identification $H^0(X, M)=H^0(Y,\sigma^*M)$. 
Note that the definition does not depend on the choice of $\sigma$.  
\item\label{vol_dfn2}
For any $\Q$-Cartier $\Q$-divisor $M$ on $X$ and for any $x\in\R_{\geq 0}$, we set 
\[
\vol_X(M-xF)
=\limsup_{\substack{k\to\infty\\kM:\text{ Cartier}}}
\frac{\dim_\Bbbk H^0(X, kM-kxF)}{k^n/n!}.
\]
The limsup is actually the limit (see \cite{L1, L2}). Moreover, the function $\vol_X(M-xF)$ 
is non-increasing and continuous over $x\in[0,\infty)$, and identically equal to zero 
for $x\gg 0$. 
\end{enumerate}
\end{definition}

\begin{definition}\label{beta_dfn}
\begin{enumerate}
\renewcommand{\theenumi}{\arabic{enumi}}
\renewcommand{\labelenumi}{(\theenumi)}
\item\label{beta_dfn1}
For any prime divisor $F$ over $X$, we set 
\[
\hat{\beta}_{(X, \Delta)}(F):=
1-\frac{\int_0^\infty\vol_X(L-xF)dx}{A_{(X, \Delta)}(F)(L^{\cdot n})}.
\]
\item\label{beta_dfn2}
(see \cite{tian87, Dem08}) We set 
\[
\alpha(X, \Delta):=\sup\{\alpha\in\Q_{\geq 0}\,|\,(X, \Delta+\alpha D)\text{: klt 
for any }D\geq 0\text{ with }D\sim_\Q L\}.
\]
\end{enumerate}
\end{definition}

The following theorem is important in this article. 

\begin{thm}[{\cite[Theorem 1.5]{fjt17}, see also \cite[Theorem 3.7]{li} 
and \cite[Theorem 6.6]{fjtb}}]\label{beta_thm}
The followings are equivalent: 
\begin{enumerate}
\renewcommand{\theenumi}{\arabic{enumi}}
\renewcommand{\labelenumi}{(\theenumi)}
\item\label{beta_thm1}
$((X, \Delta), L)$ is uniformly K-stable $($resp., K-semistable$)$, 
\item\label{beta_thm1}
there exists $\varepsilon\in(0,1)$ $($resp., $\varepsilon\in[0,1))$ such that 
$\hat{\beta}_{(X, \Delta)}(F)\geq \varepsilon$ holds for any prime divisor $F$ over $X$. 
\end{enumerate}
\end{thm}

Recently, the theory of \emph{delta-invariants} introduced in \cite{FO} is much 
developed by \cite{BJ}. The following definition is not the original definition in 
\cite{FO, BJ}. See \cite[Theorem C]{BJ} in detail. 

\begin{definition}[{\cite{FO, BJ}}]\label{delta_dfn}
We set 
\[
\delta(X, \Delta):=\inf_{\substack{F:\text{ prime}\\ \text{divisor over }X}}
\frac{1}{1-\hat{\beta}_{(X, \Delta)}(F)},
\]
and we call it the \emph{delta-invariant} of $(X, \Delta)$. From Theorem \ref{beta_thm}, 
the uniform K-stability (resp., the K-semistability) of $((X, \Delta), L)$ is equivalent to the 
condition $\delta(X, \Delta)>1$ (resp., $\delta(X, \Delta)\geq 1$). 
\end{definition}

\begin{lemma}[{see \cite[Theorem A]{BJ} and \cite[Theorem 3.5]{FO}}]\label{alpha_lemma}
We have the inequality 
\[
\alpha(X, \Delta)\geq \frac{\delta(X, \Delta)}{n+1}.
\]
\end{lemma}

\begin{proof}
We give a proof for the readers' convenience. 
Take any $D\geq 0$ with $D\sim_\Q L$. Set 
\[
c:=\max\{c'>0\,|\, (X, \Delta+c'D)\text{ is log canonical}\}.
\]
Then there exists a prime divisor $F$ over $X$ such that $A_{(X, \Delta+cD)}(F)=0$ holds. 
We remark that $\hat{\beta}_{(X, \Delta)}(F)\geq 1-\delta(X, \Delta)^{-1}$ holds. 
Take any resolution $\sigma\colon Y\to X$ with $F\subset Y$. Then 
$0=A_{(X, \Delta+cD)}(F)=A_{(X, \Delta)}(F)-c\ord_F\sigma^*D$. Thus we get 
\begin{eqnarray*}
1-\frac{1}{\delta(X, \Delta)} & \leq & 
1-\frac{\int_0^\infty\vol_Y(\sigma^*L-xF)dx}{A_{(X, \Delta)}
(F)(L^{\cdot n})}\\
&\leq&1-\frac{\int_0^\infty\vol_Y\left(\sigma^*L-\frac{cx}{A_{(X, \Delta)}(F)}\sigma^*D
\right)dx}{A_{(X, \Delta)}(F)(L^{\cdot n})}\\
&=& 1-\frac{1}{c(n+1)}.
\end{eqnarray*}
Therefore we get $c\geq\delta(X, \Delta)/(n+1)$. 
\end{proof}

We will use the following technical lemma later. 

\begin{lemma}[{\cite[Claim 2.4]{fjt17} and \cite[Theorem 6.6]{fjtb}}]\label{technical_lem}
Assume that there exists $\varepsilon\in[0,1)$ such that, for any prime divisor $F$ over 
$X$, $\hat{\beta}_{(X, \Delta)}(F)\geq \varepsilon$ holds, that is, 
$\delta(X, \Delta)\geq 1/(1-\varepsilon)$ holds. Then, for any semiample, 
normal test configuration $(\sX, \sL)/\pr^1$ of $(X, L)$, we have the inequality 
\[
\DF_\Delta(\sX, \sL)\geq \frac{\varepsilon}{n+1}\cdot J^{\NA}(\sX, \sL).
\]
\end{lemma}

The following is the main result in this section. 

\begin{thm}\label{fanolog_thm}
\begin{enumerate}
\renewcommand{\theenumi}{\arabic{enumi}}
\renewcommand{\labelenumi}{(\theenumi)}
\item\label{fanolog_thm1}
Take any $\delta_0\in\R_{>0}$ with $\delta(X, \Delta)>\delta_0$. Set 
\[
\varepsilon_0:=\frac{\delta(X, \Delta)-\delta_0}{n\cdot \delta(X, \Delta)+n+1}.
\]
Then, for any effective $\Q$-Cartier $\Q$-divisor $B$ on $X$ with $\varepsilon_0 L-B$ 
nef, $(X, \Delta+B)$ is a log Fano pair and $\delta(X, \Delta+B)\geq \delta_0$.
\item\label{fanolog_thm2}
Take any $\delta_1\in\R_{>0}$ with $\delta(X, \Delta)<\delta_1$. Set 
\[
\varepsilon_1:=\min\left\{\frac{\delta(X, \Delta)}{n+1},\quad
 1-\sqrt[n+1]{\frac{\delta(X, \Delta)}
{\delta_1}}\right\}.
\]
Then, for any effective $\Q$-Cartier $\Q$-divisor $B$ on $X$ with $\varepsilon_1 L-B$ 
ample, $(X, \Delta+B)$ is a log Fano pair and $\delta(X, \Delta+B)\leq \delta_1$.
\end{enumerate}
\end{thm}

\begin{proof}
Take any $\varepsilon\in(0,1)$ with $\varepsilon<\delta(X, \Delta)/(n+1)$. 
Assume that an effective $\Q$-Cartier $\Q$-divisor $B$ on $X$ satisfies that 
$\varepsilon L-B$ is nef. By Lemma \ref{fanocone_lem}, there exists $D\geq B$ with 
$D\sim_\R\varepsilon L$. Since $\varepsilon<1$, $L-B$ is ample. 
Take any prime divisor $F$ over $X$ and fix any resolution $\sigma\colon Y\to X$ 
with $F\subset Y$. By Lemma \ref{alpha_lemma}, we have 
\[
0\leq A_{\left(X, \Delta+\frac{\delta(X, \Delta)}{\varepsilon(n+1)}D\right)}(F)
=A_{(X, \Delta)}(F)
-\frac{\delta(X, \Delta)}{\varepsilon(n+1)}\ord_F\sigma^*D. 
\]
Thus we get 
\begin{eqnarray*}
&&A_{(X, \Delta)}(F)\geq 
A_{(X, \Delta+B)}(F)\\
&\geq& 
A_{(X, \Delta+D)}(F)\geq
\left(1-\frac{\varepsilon(n+1)}{\delta(X, \Delta)}\right)A_{(X, \Delta)}(F)>0.
\end{eqnarray*}
This implies that $(X, \Delta+B)$ is a log Fano pair. 
On the other hand, we have 
\begin{eqnarray*}
&&\vol_X(L-xF)\geq \vol_X((L-B)-xF)\\
&\geq &\vol_X((L-D)-xF)=(1-\varepsilon)^n\vol_X\left(L-\frac{x}{1-\varepsilon}F\right)
\end{eqnarray*}
for any $x\in\R_{\geq 0}$. In particular, we have 
\[
(L^{\cdot n})\geq ((L-B)^{\cdot n})\geq (1-\varepsilon)^n(L^{\cdot n}).
\]

\eqref{fanolog_thm1}
For any prime divisor $F$ over $X$, we have 
\begin{eqnarray*}
\hat{\beta}_{(X, \Delta+B)}(F)&\geq&1-\frac{\int_0^\infty\vol_X(L-xF)dx}
{\left(1-\frac{\varepsilon_0(n+1)}{\delta(X, \Delta)}\right)(1-\varepsilon_0)^n
A_{(X, \Delta)}(F)(L^{\cdot n})}\\
&\geq&1-\frac{1}{(\delta(X, \Delta)-\varepsilon_0(n+1))(1-\varepsilon_0)^n}\\
&\geq&
1-\frac{1}{\delta(X, \Delta)-\varepsilon_0(n\cdot\delta(X, \Delta)+n+1)}
=1-\frac{1}{\delta_0}.
\end{eqnarray*}

\eqref{fanolog_thm2}
For any $B$ in the assumption of \eqref{fanolog_thm2}, we can find 
$\varepsilon\in(0,\varepsilon_1)$ such that $\varepsilon L-B$ is ample. Moreover, 
for any $\delta_2\in[\delta(X, \Delta), \delta_1)$, we can find a prime divisor $F$ over $X$ 
such that $\hat{\beta}_{(X, \Delta)}(F)\leq 1-1/\delta_2$ holds. For such $F$, we have 
\begin{eqnarray*}
\hat{\beta}_{(X, \Delta+B)}(F)&\leq&1-\frac{(1-\varepsilon)^{n+1}
\int_0^\infty\vol_X(L-xF)dx}
{A_{(X, \Delta)}(F)(L^{\cdot n})}\\
&\leq&1-\frac{(1-\varepsilon)^{n+1}}{\delta_2}
\leq 1-\frac{1}{\delta_1\delta_2/\delta(X, \Delta)}.
\end{eqnarray*}
Thus $\delta(X, \Delta+B)\leq\delta_1\delta_2/\delta(X, \Delta)$ for any 
$\delta_2\in[\delta(X, \Delta), \delta_1)$. Therefore, we get the inequality 
$\delta(X, \Delta)\leq \delta_1$.
\end{proof}

\begin{proof}[Proof of Theorem \ref{fanolog_intro_thm}]
If $\varepsilon L-B$ is nef, then $((X, \Delta+B), L-B)$ is K-semistable by Theorem 
\ref{fanolog_thm} \eqref{fanolog_thm1}. 
Assume that $\varepsilon L-B$ is ample. 
We can take $\delta_0\in(1, \delta(X, \Delta))$ such that $\varepsilon_0 L-B$ is ample, 
where $\varepsilon_0:=(\delta(X, \Delta)-\delta_0)/(n\delta(X, \Delta)+n+1)$. 
Now Theorem \ref{fanolog_intro_thm} is an immediate consequence of Theorem 
\ref{fanolog_thm} \eqref{fanolog_thm1}. 
\end{proof}

\begin{corollary}\label{fanolog_cor}
Fix any norm $\|\cdot\|$ on $\ND(X)_\R$. Take any $\delta_0$, $\delta_1\in\R_{>0}$ 
with $\delta(X, \Delta)\in(\delta_0,\delta_1)$. Then there exists $\varepsilon\in\R_{>0}$ 
such that $(X, \Delta+B)$ is a log Fano pair with 
$\delta(X, \Delta+B)\in(\delta_0,\delta_1)$ for any effective $\Q$-Cartier 
$\Q$-divisor $B$ on $X$ 
with $\| B\|\leq \varepsilon$. 
\end{corollary}

\begin{proof}
Follows immediately from Theorem \ref{fanolog_thm}. 
\end{proof}

\section{Perturbing boundaries}\label{perturb_section}

In this section, we prove Theorem \ref{mainthm}. 
Technically, the following proposition is important in this article. 

\begin{proposition}\label{perturb_prop}
Let $(X, \Delta)$ be an $n$-dimensional projective demi-normal pair, $L$ be an ample 
$\Q$-line bundle and $N$ be an effective and nef $\Q$-divisor on $X$. Then, for any 
semiample demi-normal test configuration $(\sX,\sL)/\pr^1$ of $(X, L)$, we have 
\[
n\mu_N(L) J^{\NA}(\sX, \sL)\geq \DF_{\Delta+N}(\sX, \sL)-\DF_\Delta(\sX, \sL). 
\]
\end{proposition}

\begin{proof}
Let 
\[\xymatrix{
& \sY  \ar[dl]_\Theta \ar[dr]^\Pi & \\
\sX \ar@{-->}[rr] & & X\times\pr^1
}\]
be the partial normalization of the graph and let $p_1\colon X\times\pr^1\to X$ be the 
first projection. We write $\phi:=\Theta^*\sL$, $\psi_L:=\Pi^*p_1^*L$, 
$\psi_N:=\Pi^*p_1^*N$ for simplicity. 
By Lemma \ref{negativity_lem}, we have 
\[
(\phi^{\cdot j}\cdot\psi_L^{\cdot n-2-j}\cdot (\phi-\psi_L)^{\cdot 2}\cdot\psi_N)\leq 0
\]
for any $0\leq j\leq n-2$. Thus we get $(\phi^{\cdot n}\cdot \psi_N)\leq 
n(\phi\cdot\psi_L^{\cdot n-1}\cdot\psi_N)$. 
Moreover, we note that 
\begin{eqnarray*}
&&(L^{\cdot n})\left(\DF_{\Delta+N}(\sX, \sL)-\DF_\Delta(\sX, \sL)\right)\\
&=&(\phi^{\cdot n}\cdot N_\sY)-\frac{n}{n+1}\mu_N(L)(\phi^{\cdot n+1})\\
&\leq&(\phi^{\cdot n}
\cdot \psi_N)-\frac{n}{n+1}\mu_N(L)(\phi^{\cdot n+1})
\end{eqnarray*}
since $N$ is effective. We also note that 
\begin{eqnarray*}
&&(\phi^{\cdot n}
\cdot \psi_N)-\frac{n}{n+1}\mu_N(L)(\phi^{\cdot n+1})\\
&\leq&n((\phi\cdot\psi_L^{\cdot n-1}\cdot\psi_N)-\mu_N(L)(\phi\cdot\psi_L^{\cdot n}))
+n\mu_N(L)(L^{\cdot n})J^{\NA}(\sX, \sL)\\
&=&n\mu_N(L)(L^{\cdot n})J^{\NA}(\sX, \sL)
\end{eqnarray*}
since $\psi_L^{\cdot n-1}\cdot(\psi_N-\mu_N(L)\psi_L)\equiv 0$ as a $\Q$-1-cycle. 
\end{proof}

As consequences of Proposition \ref{perturb_prop}, we have many results. 
The following is a baby version of Theorem \ref{W_thm}. 

\begin{corollary}\label{curve_cor}
Let $(X, \Delta)$ be a $1$-dimensional projective slc pair such that $K_X+\Delta$ is 
ample. Then $((X, \Delta), L)$ is uniformly K-stable for any ample $\Q$-line bundle $L$. 
\end{corollary}

\begin{proof}
We may assume that $L-(K_X+\Delta)$ is ample (by replacing $L$ with high multiple). 
Take a general $\Q$-divisor $A\geq 0$ with $A\sim_\Q L-(K_X+\Delta)$. Then 
$(X, \Delta+A)$ is an slc pair. Thus, by Theorem \ref{odk_thm} \eqref{odk_thm1}, for any 
semiample, demi-normal test configuration $(\sX, \sL)/\pr^1$ of $(X, L)$, we have 
$\DF_{\Delta+A}(\sX, \sL)\geq J^{\NA}(\sX, \sL)$. On the other hand, by Proposition 
\ref{perturb_prop}, we have $\DF_\Delta(\sX, \sL)\geq (1-\mu_A(L))J^{\NA}(\sX, \sL)$. 
Thus we get the assertion since $1-\mu_A(L)=\mu_{K_X+\Delta}(L)>0$. 
\end{proof}

\begin{corollary}\label{fano1_cor}
Let $(X, \Delta)$ be an $n$-dimensional log Fano pair. Set $L:=-(K_X+\Delta)$. 
Assume that $\delta(X, \Delta)>1$. Take any $\delta\in(1, \delta(X, \Delta))$. 
Set
\[
\varepsilon:=\frac{\delta(X, \Delta)-\delta}{n\cdot\delta(X, \Delta)+n+1}, \quad\quad
\delta_1:=\frac{\delta-1}{(n+1)\delta}.
\]
Take any nef $\Q$-divisor $N$ on $X$ with 
$\varepsilon L-N$ nef. Then, for any semiample, normal test 
configuration $(\sX, \sL')/\pr^1$ of $(X, L-N)$, we have 
\[
\DF_\Delta(\sX, \sL')\geq(\delta_1-n\mu_N(L-N)) J^{\NA}(\sX, \sL').
\] 
$($In particular, if $\delta_1>n\mu_N(L-N)$ $($resp., if $\delta_1\geq n\mu_N(L-N)$$)$, 
then $((X, \Delta), L-N)$ is uniformly K-stable $($resp., K-semistable$)$.$)$
\end{corollary}

\begin{proof}
By Lemma \ref{fanocone_lem}, we may assume that $N$ is effective. By 
Theorem \ref{fanolog_thm} \eqref{fanolog_thm1}, $(X, \Delta+N)$ is a log Fano pair 
with $\delta(X, \Delta+N)\geq \delta$. 
Thus we get  
\begin{eqnarray*}
&&\delta_1\cdot J^{\NA}(\sX, \sL')\leq \DF_{\Delta+N}(\sX, \sL')\\
&\leq&\DF_\Delta(\sX, \sL')+n\mu_N(L-N)\cdot J^{\NA}(\sX, \sL').
\end{eqnarray*}
from Lemma \ref{technical_lem} and Proposition 
\ref{perturb_prop}.
\end{proof}

\begin{corollary}\label{fano2_cor}
Let $(X, \Delta)$ be an $n$-dimensional log Fano pair and set $L:=-(K_X+\Delta)$. 
Assume that $\delta(X, \Delta)>1$. Set 
\[
\varepsilon:=\frac{\delta(X, \Delta)-1}{(n^2+n+1)\delta(X, \Delta)+n^2+n-1}.
\]
Take any nef $\Q$-divisor $N$ on $X$ with $\varepsilon L-N$ ample. Then 
$((X, \Delta), L-N)$ is uniformly K-stable. In particular, by Proposition \ref{norm_prop}, 
for any norm $\|\cdot\|$ on $\ND(X)_\R$, 
there exists $\varepsilon'\in(0,1)$ such that $((X, \Delta), L')$ is uniformly K-stable 
for any $\Q$-line bundle $L'$ on $X$ with $\|L'-L\|\leq \varepsilon'$. 
\end{corollary}

\begin{proof}
We may assume that $n\geq 2$. We just apply Corollary \ref{fano1_cor} for 
$\delta:=(\delta(X, \Delta)+1)/2$. Note that 
\begin{eqnarray*}
&&\left((n^2+n+1)\delta(X, \Delta)+n^2+n-1\right)
-2\left(n\cdot\delta(X, \Delta)+n+1\right)\\
&=&(n^2-n+1)\delta(X, \Delta)+n^2-n-3\geq n^2-n\geq 0
\end{eqnarray*}
since $n\geq 2$. Thus we get 
\[
\varepsilon\leq\frac{\delta(X, \Delta)-1}{2\left(n\cdot\delta(X, \Delta)+n+1\right)}
=\frac{\delta(X, \Delta)-\delta}{n\cdot\delta(X,\Delta)+n+1}.
\]
Moreover, the condition $\mu_N(L-N)<(\delta-1)/(n(n+1)\delta)$ is equivalent to the 
condition 
\[
\left((L-N)^{\cdot n-1}\cdot\left(\frac{\delta-1}{(n^2+n+1)\delta-1}L-N\right)\right)>0.
\]
Note that $(\delta-1)/((n^2+n+1)\delta-1)$ is equal to $\varepsilon$.
\end{proof}

\begin{thm}[{cf.\ \cite{weinkove, song-weinkove}}]\label{W_thm}
Let $(X, \Delta)$ be an $n$-dimensional projective slc pair with $n\geq 2$ 
and let $L$ be an ample 
$\Q$-line bundle on $X$. Assume that $\mu_{K_X+\Delta}(L)>0$ and 
\[
\frac{n^2}{n^2-1}\mu_{K_X+\Delta}(L)L-(K_X+\Delta)
\]
is ample $($resp., nef$)$. Then $((X, \Delta), L)$ is uniformly K-stable $($resp., 
K-semistable$)$. 
\end{thm}

\begin{proof}
Take any $\varepsilon\in \Q$ with $0<\varepsilon\ll 1$ (resp., $-1\ll\varepsilon <0$) 
such that 
\[
\left(\frac{n^2}{n^2-1}\mu_{K_X+\Delta}(L)-\varepsilon\right)L-(K_X+\Delta)
\]
is ample. Take a general effective $\Q$-divisor 
$A$ with small coefficients $\Q$-linearly equivalent to $((n^2/(n^2-1))\mu_{K_X+\Delta}(L)-\varepsilon)L-(K_X+\Delta)$. Then $(X, \Delta+A)$ is an slc pair and 
\[
K_X+\Delta+A\sim_\Q\left(\frac{n^2}{n^2-1}\mu_{K_X+\Delta}(L)-\varepsilon\right)L.
\]
Thus, for any semiample, demi-normal test configuration $(\sX, \sL)/\pr^1$ of $(X, L)$, 
we have 
\[
\DF_{\Delta+A}(\sX, \sL)\geq \frac{1}{n}\left(\frac{n^2}{n^2-1}\mu_{K_X+\Delta}(L)-\varepsilon\right)J^{\NA}(\sX, \sL)
\]
by Theorem \ref{odk_thm} \eqref{odk_thm1}. On the other hand, by Proposition 
\ref{perturb_prop}, we have 
\[
\DF_{\Delta+A}(\sX, \sL)-\DF_\Delta(\sX, \sL)\leq n\mu_A(L)\cdot J^{\NA}(\sX, \sL). 
\]
Therefore we get 
\[
\DF_\Delta(\sX, \sL)\geq \left(n-\frac{1}{n}\right)\varepsilon\cdot J^{\NA}(\sX, \sL)
\]
by combining those inequalities. 
\end{proof}

\begin{remark}\label{CD_rmk}
\begin{enumerate}
\renewcommand{\theenumi}{\arabic{enumi}}
\renewcommand{\labelenumi}{(\theenumi)}
\item\label{CD_rmk1}
Assume that $\mu_{K_X+\Delta}(L)=0$ and $-(K_X+\Delta)$ is nef. 
Then, Dervan pointed out to the author that, we can easily show that 
$K_X+\Delta$ is numerically trivial. When $K_X+\Delta$ is numerically trivial, the uniform 
K-stability and the K-semistability of $((X,\Delta), L)$ is well-understood 
by Theorem \ref{odk_thm} \eqref{odk_thm2}.
\item\label{CD_rmk2}
The author found Theorem \ref{W_thm} under the additional hypothesis 
``$K_X+\Delta$ is ample" in order to prove Corollary \ref{gt2_cor}. 
Codogni and Dervan pointed out to the author 
that we do not need the assumption. 
\end{enumerate}
\end{remark}

\begin{corollary}\label{gt1_cor}
Let $(X, \Delta)$ be an $n$-dimensional projective slc pair such that $L:=K_X+\Delta$ 
is ample. Then $((X, \Delta), L+N)$ is uniformly K-stable for any nef $\Q$-divisor $N$ with 
$L-(n^2-1)N$ big. 
\end{corollary}

\begin{proof}
We may assume that $n\geq 2$ by Corollary \ref{curve_cor}. 
Set $M:=L+N$. Since $M-n^2N$ is big, we have $\mu_N(M)<1/n^2$. Note that 
\[
\frac{n^2}{n^2-1}\mu_L(M)M-L=\frac{1}{n^2-1}(1-n^2\mu_N(M))M+N
\]
is ample. Thus $((X, \Delta), L+N)$ is uniformly K-stable 
by Theorem \ref{W_thm}. 
\end{proof}

\begin{corollary}\label{gt2_cor}
Let $(X, \Delta)$ be a projective slc pair such that $L:=K_X+\Delta$ is ample. 
Fix any norm $\|\cdot\|$ on $\ND(X)_\R$. Then there exists $\delta>0$ such that 
$((X, \Delta), L')$ is uniformly K-stable for any $\Q$-line bundle $L'$ on $X$ 
with $\|L'-L\|\leq\delta$. 
\end{corollary}

\begin{proof}
Follows immediately from Corollaries \ref{curve_cor}, 
\ref{gt1_cor} and Proposition \ref{norm_prop}.
\end{proof}

\begin{proof}[Proof of Theorem \ref{mainthm}]
Follows immediately from Corollaries \ref{fano2_cor}, \ref{gt2_cor} and Theorem 
\ref{odk_thm} \eqref{odk_thm3}. 
\end{proof}

\end{document}